\documentclass[11pt,a4paper]{amsart}
\usepackage[dvipdfmx]{graphicx}
\usepackage[dvipdfmx]{color}
\usepackage[dvipdfmx,colorlinks=true]{hyperref}

\usepackage{mathtools}

\makeatletter
\newtheorem*{rep@theorem}{\rep@title}
\newcommand{\newreptheorem}[2]{%
\newenvironment{rep#1}[1]{%
 \def\rep@title{#2 \ref{##1}}%
 \begin{rep@theorem}}%
 {\end{rep@theorem}}}
\makeatother

\definecolor{RedOrange}{cmyk}{ 0, 0.77, 0.87, 0}
\definecolor{RoyalPurple}{cmyk}{ 0.84, 0.53, 0, 0}
\definecolor{YellowGreen}{cmyk}{ 0.44, 0, 0.74, 0}
\definecolor{Fuchsia}{cmyk}{ 0.47, 0.91, 0, 0.08}
\definecolor{Blue}{cmyk}{ 0.84, 0.53, 0, 0}
\definecolor{BlueViolet}{cmyk}{ 0.84, 0.53, 0, 0}
\definecolor{Black}{cmyk}{ 0.75, 0.68, 0.67, 0.9}
\usepackage{xcolor}
\usepackage{verbatim}
\usepackage{amsmath}
\usepackage{amssymb, mathrsfs}
\usepackage{amsbsy}

\usepackage{amscd}
\usepackage{amsthm}

\usepackage[english]{babel}
\usepackage{todonotes}
\usepackage[T1]{fontenc}

\usepackage{color}

\newcommand{\lf}{\lfloor}
\newcommand{\rf}{\rfloor}
\newcommand{\R}{\mathbb{R}}

\newcommand{\N}{\mathbb{N}}
\newcommand{\e}{\varepsilon}
\newcommand{\la}{\langle}
  \newcommand{\ra}{\rangle}
\newcommand{\E}{\mathbb{E}}
\newcommand{\Z}{\mathbb{Z}}

\renewcommand{\S}{\mathbb{S}}
\renewcommand{\P}{\mathbb{P}}
\renewcommand{\L}{\mathbb{L}}

\newcommand{\rmT}{\mathrm{T}}
\newcommand{\rmB}{\mathrm{B}}
\newcommand{\rmD}{\mathrm{D}}

\newcommand{\rmp}{\mathrm{p}}

\newcommand{\lin}{\left[\kern-0.15em\left[}
\newcommand{\rin} {\right]\kern-0.15em\right]}
\newcommand{\linf}{[\kern-0.15em [}
\newcommand{\rinf} {]\kern-0.15em ]}
\newcommand{\ilin}{\left]\kern-0.15em\left]}
\newcommand{\irin} {\right[\kern-0.15em\right[}

\def\ben#1{\begin{equation}#1\end{equation}}

\def\bean#1{\begin{eqnarray}#1\end{eqnarray}}
\def\al#1{\begin{align*}#1\end{align*}}
\def\aln#1{\begin{align}#1\end{align}}

\usepackage{constants}

\newconstantfamily{c}{symbol=c}

\newconstantfamily{a}{symbol=\alpha}

\newcommand{\secno}[1]{\thesection.\arabic{#1}}
\newconstantfamily{kE}{
symbol=\mathcal{E},
format=\secno,
reset={section}
}

\newtheorem{lem}{Lemma}[section]

\newtheorem{prop}[lem]{Proposition}
\newtheorem{thm}[lem]{Theorem}

\usepackage{color}
\definecolor{lilas}{RGB}{182, 102, 210}

\numberwithin{equation}{section}

\def\ben#1{\begin{equation}#1\end{equation}}

\def\bean#1{\begin{eqnarray}#1\end{eqnarray}}

\title[Upper tail LDP for a class of distributions in FPP]
{Upper tail large deviations for a class of distributions in First-passage percolation}
\date{\today}
\author{Shuta Nakajima} 
\address[Shuta Nakajima]
{Graduate School of Mathematics, University Nagoya.}
\email{njima@math.nagoya-u.ac.jp}


\begin{document}

\maketitle

\begin{abstract}
For first passage percolation with identical and independent exponential distributions, called the Eden growth model, we study the upper tail large deviations for the first passage time $\rmT$. In this paper we show that for any $\xi>0$ and $x\neq 0$, $\P(\rmT(0,nx)>n(\mu+\xi))$ decays as $\exp{(-(2d\xi +o(1))n)}$ with a time constant $\mu$ and a dimension $d$. Moreover, we extend the result to stretched exponential distributions. On the contrary, we construct a continuous distribution with a finite exponential moment where the rate function does not exist. 
\end{abstract}
\section{Introduction}
\subsection{Introduction and Main results}
First-passage percolation (FPP) was first introduced by Hammersley and Welsh in 1965, as a dynamical version of the percolation model \cite{HW65}. Since then, it has been extensively studied both in mathematics and physics. There are several reasons that FPP is an attractive model in the fields. One reason is that FPP naturally defines a random metric space. Indeed, the objects of important in FPP, called the first passage time and the optimal path, correspond to a metric and a geodesic in a certain random metric space. Another is that FPP is expected to belong to the KPZ universality class. Moreover, it is widely believed that the boundary of a ball defined by the first passage time behaves like a KPZ equation \cite{KPZ86,KS91}. See \cite{ADH} for more detailed backgrounds.\\

In this paper, we consider the FPP on the lattice $\L^d=(\Z^d,\E^d)$ with $d\geq 2$. The model is defined as follows.  To each edge $e\in \E^d$, we assign a non-negative random variable $\tau_e$. We assume that the collection $\tau=(\tau_e)_{e\in\E^d}$ is identically and independent distributions and there exist $c_1,c_2,\alpha>0$ and $r\in(0,1]$ such that for sufficiently large $t>0$
  \ben{\label{cond:distr}
    \P(\tau_e=0)<\rmp_c(d)\text{ and }c_1\exp{(-\alpha t^r)}\leq \P(\tau_e>t)\leq c_2\exp{(-\alpha t^r)},
  }
  where $\rmp_c(d)$ stands for the critical probablity of $d$-dimensional percolation. Note that for exponential distributions, the models are specially called the Eden growth model \cite{Eden}. For the generalization of the condition \eqref{cond:distr}, see Section~\ref{remark:1}.\\

  A sequence $(x_i)_{i=1}^{l}$ is said to be a path if each successive pair is nearest neighbor, i.e. $|x_i-x_{i+1}|_1=1$ for any $i$. We note that a path is seen both as a set of vertices and a set of edges with some abuse of notation. Given a path $\gamma$, we define the passage time of $\gamma$ as
$$\rmT(\gamma)=\sum_{e\in\gamma}\tau_e.$$
For $x\in\R^d$, we set $\lf x\rf=(\lf x_1\rf,\cdots,\lf x_d\rf)$  where $\lf a\rf$ is the greatest integer less than or equal to $a$ for $a\in\R$. Given $x,y\in\R^d$, we define the {\em first passage time} between $x$ and $y$ as
$$\rmT(x,y)=\inf_{\gamma:\lf x\rf\to \lf y\rf}\rmT(\gamma),$$
where the infimum is taken over all finite paths $\gamma$ starting at $\lf x\rf$ and ending at $\lf y\rf$.\\

By Kingman's subadditive ergodic theorem \cite{King73}, if $\E \tau_e<\infty$, for any $x\in\R^d$, there exists an non-random constant $\mu(x)\ge 0$ (called the {\em time constant}) such that

\begin{align}\label{kingman}
  \mu(x)=\lim_{t\to\infty}t^{-1} \rmT(0,t x)=\lim_{t\to\infty}t^{-1} \E[\rmT(0,t x)]\hspace{4mm}a.s.
\end{align}
This convergence can be seen as a law of large numbers. The following result shows the corresponding upper large deviations.
\begin{thm}\label{thm:main1}
  For any $\xi>0$ and $x\in\R^d\backslash\{0\}$,
  \ben{
    \lim_{n\to\infty}\frac{1}{n^r}\log{\P(\rmT(0,nx)>n(\mu(x)+\xi))}=-2d\alpha \xi^r,\label{thm:maineq1}
  }  
  where $\alpha$ and $r$ are in \eqref{cond:distr}.
  \end{thm}
Next, we construct a distribution where the rate function does not exist. Let $a_0=0$ and we define a sequence $a_n$ inductively as
$$a_{n+1}=2^{a_n}.$$
Let $\alpha_1<\alpha_2$ be positive constants. We now consider a non-negative distribution satisfying
\ben{\label{anomaly:distr}
  \P(\tau_e\in dx)=
  \begin{cases}
    c_3\exp{(-\alpha_2 x)}dx,&\text{ if $x\in[a_{2n},a_{2n+1})$ with some $n\in\Z_{\geq{}0}$}\\
      c_3\exp{(-\alpha_1 x)}dx,&\text{ if $x\in[a_{2n-1},a_{2n})$ with some $n\in\Z_{\geq{}0}$},
  \end{cases}
      }
    where $c_3>0$ is determined so that $\int_{0}^{\infty}\P(\tau_e\in dx)=1$. It is straightforward to check that
    $$\E[\exp{(-\rho\tau_e)}]<\infty\text{ for any $\rho<\alpha_1$}.$$
    \begin{thm}\label{thm:main3}
      Suppose \eqref{anomaly:distr}. Then for any $\xi>0$,
      \ben{\label{limsup}
        \limsup_{n\to\infty}\frac{1}{n}\log{\P(\rmT(0,n\mathbf{e}_1)>n(\mu(\mathbf{e}_1)+\xi))}= -2d\alpha_1 \xi,
      }
      \ben{\label{liminf}
        \liminf_{n\to\infty}\frac{1}{n}\log{\P(\rmT(0,n\mathbf{e}_1)>n(\mu(\mathbf{e}_1)+\xi))}= -2d\alpha_2 \xi.
      }
      In particular,
      \al{
      -\infty&<\liminf_{n\to\infty}\frac{1}{n}\log{\P(\rmT(0,n\mathbf{e}_1)>n(\mu(\mathbf{e}_1)+\xi))}\\
      &<\limsup_{n\to\infty}\frac{1}{n}\log{\P(\rmT(0,n\mathbf{e}_1)>n(\mu(\mathbf{e}_1)+\xi))}<0.
      }
      \end{thm}
\subsection{Related work}
Large deviation principle (LDP) is one of the major subjects in probability theory. The study of the large deviations in the first passage percolation was initiated by Kesten \cite{Kes86}. For the lower tail LDP, by using the usual subadditivity argument, he obtained that for $\xi>0$, the following limit exists and is negative:
\ben{\label{Kes: lower}
\lim_{n\to\infty}\frac{1}{n}\log{  \P(\rmT(0,n\mathbf{e}_1)<n(\mu(\mathbf{e}_1)-\xi))}.
  }
On the other hand, he showed that under the boundedness of the distrbution,
\aln{
  -\infty&<\liminf_{n\to\infty}\frac{1}{n^d}\log{  \P(\rmT(0,n\mathbf{e}_1)>n(\mu(\mathbf{e}_1)+\xi))}\nonumber\\
  &\leq \limsup_{n\to\infty}\frac{1}{n^d}\log{  \P(\rmT(0,n\mathbf{e}_1)>n(\mu(\mathbf{e}_1)+\xi))}<0. \label{Kes: upper}
  }
It is worth noting that the rate of upper LDP and lower LDP are different (See \cite{CZ03} for the heuristics of the difference). Although the existence of the rate function of upper tail LDP had been an open problem for many years, the authors in \cite{BGS} solved this recently under the boundedness of the distribution with some assumption of continuity. The assumption of boundedness is essential since the rate of \eqref{Kes: upper} may change for general distribution \cite{CGM09}. For example, for exponential distributions, it was proved that
\al{
  -\infty&<\liminf_{n\to\infty}\frac{1}{n}\log{  \P(\rmT(0,n\mathbf{e}_1)>n(\mu(\mathbf{e}_1)+\xi))}\\
  &\leq \limsup_{n\to\infty}\frac{1}{n}\log{  \P(\rmT(0,n\mathbf{e}_1)>n(\mu(\mathbf{e}_1)+\xi))}<0.
}
The different scalings appear due to different pictures of the upper LDP event \linebreak$\{\rmT(0,nx)>(\mu(x)+\xi)n\}$. Indeed, for bounded distributions, the upper LDP events are   affected by overall configurations. In contrast, for exponential distributions, the upper LDP events highly depend on the configurations around the starting point and the ending point. Hence, for general distibution, we need to take a distribution-dependent approach to study upper LDP in depth. In this paper, we consider exponential and stretched exponential distributions. In these cases, studying the neighborhood of the endpoints carefully, we can do a more detailed analysis, which enables us to get the exact value of the rate function.\\

Incidentally, in the frog models, it is proved that $-\log{\P(\rmT(0,nx)>(\mu_f(x)+\xi)n)}$, with a certain time constant $\mu_{f}(x)$, grows like (i)  $\sqrt{n}$ for $d=1$; (ii) $n/\log n$ for $d=2$; (iii) $n$ for $d\geq 3$ \cite{CKN}. Hence, it would be interesting if one identifies the rate functions of the frog models as in our results.

\subsection{Remark on the generalization of the condition \eqref{cond:distr}}\label{remark:1}
  For $r<1$, we can weaken the second condition of \eqref{cond:distr} as follows. Suppose that there exist slowly varying functions $c_1(t),c_2(t),{\rm b}_1(t),{\rm b}_2(t)$ and $0<r<1$ such that $\displaystyle\lim_{t\to\infty}\frac{{\rm b}_1(t)}{{\rm b}_2(t)}=1$ and for $t>0$,
  \ben{\label{cond-distr2}
  \P(\tau_e=0)<{\rm p}_c(d)\text{ and } c_1(t)\exp{(-{\rm b}_1(t)t^r)}\leq \P(\tau_e>t)\leq c_2(t)\exp{(-{\rm b}_2(t)t^r)},
  }
  where a function $f(t)$ is said to be slowly varying if for any $a>0$, $f(a)>0$ and
  $$\lim_{t\to\infty}\frac{f(at)}{f(t)}=1.$$
  Then, \eqref{thm:maineq1} is replaced by the following:
  \ben{\label{thm:main2}
    \lim_{n\to\infty}\frac{1}{{\rm b}_1(n)n^r}\log{\P(\rmT(0,nx)>n(\mu(x)+\xi))}=-2d \xi^r.
    }
  See Section~\ref{slow-vary} for their proofs.
\subsection{Notation and terminology}
This subsection collects some useful notations, terminologies and remarks.
\begin{itemize}
  \item Given nearest neighbor vertices $v,w\in\Z^d$, we write $\la v,w \ra$ for the edge connecting $v$ and $w$.
 \item Given two vertices $v,w\in\Z^d$ and a set $D\subset\Z^d$, we set the {\em restricted} first passage time as
$$\rmT_D(v,w)=\inf_{\gamma\subset D}\rmT(\gamma),$$
   where the infimum is taken over all paths $\gamma$ from $v$ to $w$ with $\gamma\subset D$. If such a path does not exist, we set it to be infinity instead.
 \item Similarly, given a set $\E\subset \E^d$, we define
   $$\rmT_{\E}(v,w)=\inf_{\gamma\subset \E}\rmT(\gamma).$$
 \item Given $K,M\in\N$, we define
   $$\rmB_{K,M}=3K\Z^{d-1}\cap [-M,M]^{d-1}.$$
 \item Given $v\in \rmB_{K,M}$ and $n\in\N$, we define a slab as
   $$\S_v(K,n)=\{(x_i)_{i=1}^d\in\Z^d:~0\leq x_1\leq n,~(x_2,\cdots,x_d)\in (v+[-K,K]^{d-1})\}.$$
 \item Given $k\in\N$ and $x\in\Z^d$, we define
   $$\rmD_k(x)=x+[-k,k]^d.$$
  \item Given $(x_i)_{i=1}^k\subset \R$ and randoom variables $(X_i)_{i=1}^k$ with $k\in\N$,
   \al{\P\left(\sum_{i=1}^k X_i\geq \sum_{i=1}^k x_i\right)&\leq \P\left(\exists i\in\{1,\cdots,k\}\text{ s.t. } X_i\geq  x_i\right)\\
   &\leq \sum_{i=1}^k\P(X_i\geq x_i).
 }
   We use this inequality throughout this paper without any comment.
\end{itemize}
\section{Proofs}
\subsection{Proof of Theorem~\ref{thm:main1}}
For the simplicity of notation, we only consider the case $x=\mathbf{e}_1$, though the same proof works for general $x$. We write $\rmT_n=\rmT(0,n\mathbf{e}_1)$ and \allowbreak$\mu=\mu(\mathbf{e}_1)$.\\

We start with the lower bound. We fix $\xi>0$ and take $\e\in(0,\xi)$ arbitrary. Let
$$\E_1=\{e\in\E^d:~0\in e\}\text{ and }\L_1=\{x\in\Z^d:~|x|_1=1\}.$$
Roughly speaking, the lower bound comes from the event $\{\forall e\in\E_1,~\tau_e\geq \xi n\}$ whose probability is approximately $\exp{(-2d\alpha \xi^r n^r)}$. Indeed, on this event, $\{\rmT_n>(\mu+\xi)n\}$ is likely to occur. We make the heuristics rigorous. If for any $x\in \L_1$,
$$\tau_{\la0,x\ra}>(\xi+\e)n\text{ and }\rmT_{\E^d\backslash \E_1}(x,n\mathbf{e}_1)>(\mu-\e)n,$$
 then since
$$\rmT_n=\inf_{x\in\L_1}(\tau_{\la0,x\ra}+\rmT(x,n\mathbf{e}_1))\geq \inf_{x\in\L_1}\tau_{\la0,x\ra}+\inf_{x\in\L_1}\rmT_{\E^d\backslash \E_1}(x,n\mathbf{e}_1), $$
we get $\rmT_n>(\mu+\xi)n$. Thus
\aln{
  \P(\rmT_n>(\mu+\xi)n)&\geq \P(\forall x\in \L_1,~\tau_{\la0,x\ra}>(\xi+\e)n,~\rmT_{\E^d\backslash \E_1}(x,n\mathbf{e}_1)>(\mu-\e)n)\nonumber\\
  &= \P(\forall x\in\L_1,~\tau_{\la0,x\ra}>(\xi+\e)n)\P(\forall x\in\L_1,~\rmT_{\E^d\backslash \E_1}(x,n\mathbf{e}_1)>(\mu-\e)n).\label{lower-est}
}
The first term can be bounded from below by  $c_1^{2d}\exp{(-2d\alpha (\xi+\e)^r n^r)}$, where $c_1$ is in \eqref{cond:distr}. On the other hand, for the second term, since $$\rmT_n\leq\max_{e\in\E_1}\tau_e+\min_{x\in\L_1}\rmT_{\E^d\backslash \E_1}(x,n\mathbf{e}_1),$$
we obtain
\al{
  &\quad \P(\forall x\in\L_1,~\rmT_{\E^d\backslash \E_1}(x,n\mathbf{e}_1)>(\mu-\e)n)\\
  &\geq \P\left(\forall e\in\E_1,~\tau_e<\frac{\e n}{2},~\rmT_n>\left(\mu-\frac{\e}{2}\right)n\right)\\
  &\geq \P\left(\rmT_n>\left(\mu-\frac{\e}{2}\right)n\right)-\P\left(\exists e\in\E_1,~\tau_e\geq \frac{\e n}{2}\right),
}
which converges to $1$ as $n\to\infty$. Therefore, for sufficiently large $n$, we have
\ben{\label{lower-main}
  \P(\rmT_n>(\mu+\xi)n)\geq \frac{c_1^{2d}}{2}\exp{(-2d\alpha (\xi+\e)^r n^r)}.
  }
Since $\e$ is arbitrary, letting $\e\to 0$ after $n\to\infty$, we get
$$\liminf_{n\to\infty}\frac{1}{n^r}\log{\P(\rmT_n>(\mu+\xi)n)}\geq -2d\alpha \xi^r.$$
Next, we move on to the upper bound. The following lemma is a variant of \cite[Lemma 3.1]{CZ03}. In fact, the case $r=1$ is proved there. We prove it in Appendix.
\begin{lem}\label{lem:Zhang}
 For any $\e>0$, there exist $K=K(\e)\in\N$ and a positive constant $c=c(\e,K)$ such that $n\geq K$,
  $$\P\left(\rmT_{[0,n]\times[-K,K]^{d-1}}\left(0,n\mathbf{e}_1\right)\geq (\mu+\e)n\right)\leq\exp{(-cn^r)}.$$
\end{lem}
Let $\e\in(0,\xi)$ and we take such $K\in\N$ and $c>0$. Let $M=M(\xi,\e,c,K)\in 3K\N$ so that
\ben{\label{choice:M} \frac{cM}{K}>12d\alpha\xi^r.}
For simplicity, we write $\S_v=\S_v(K,n)$. Given $v\in\Z^{d-1}$, we write
$$v^{[1]}_n=(0,v)\text{ and }v^{[2]}_n=(n,v).$$
Note that $v^{[1]}_n\in \rmD_M(0)$ and $v^{[2]}_n\in \rmD_M(n\mathbf{e}_1)$ for $v\in \rmB_{K,M}$.  For $v\neq w\in \rmB_{K,M}$, since $\S_v$ and $\S_w$ are disjoint, $\rmT_{\S_v}\left(v^{[1]}_n,v^{[2]}_n\right)$ and $\rmT_{\S_w}\left(w^{[1]}_n,w^{[2]}_n\right)$ are independent. Moreover, for $v\in \rmB_{K,M}$, $\S_v$ is congruent with $[0,n]\times[-K,K]^{d-1}$. Before going into the proof, we briefly explain the heuristics of the proof. By Lemma~\ref{lem:Zhang}, with high probability, we can find $v\in \rmB_{K,M}$ such that
$$\rmT_{\S_v}\left(v^{[1]}_n,v^{[2]}_n\right)\leq (\mu+\e)n.$$
By the triangular inequality,
\al{
  \rmT_n&\leq \rmT(0,v^{[1]}_n)+\rmT_{\S_v}\left(v^{[1]}_n,v^{[2]}_n\right)+\rmT\left(v^{[2]}_n,n\mathbf{e}_1\right)\\
  &\leq \rmT(0,v^{[1]}_n)+\rmT\left(v^{[2]}_n,n\mathbf{e}_1\right)+(\mu+\e)n.
  }
Thus, $\rmT_n> (\mu+\xi)n$ implies $\rmT\left(0,v^{[1]}_n\right)+\rmT\left(v^{[2]}_n,n\mathbf{e}_1\right)\geq (\xi-\e)n$. In order to estimate the latter event, we appeal to the large deviations of the sum of independent random variables (See Lemma~\ref{exp:est} below). Let us make the above heuristics rigorous.  By using Lemma~\ref{lem:Zhang} and \eqref{choice:M}, since $\sharp \rmB_{K,M}\geq M/3K$, for $n>K$,
  \al{
\P\left(\forall v\in \rmB_{K,M},~\rmT_{\S_v}\left(v^{[1]}_n,v^{[2]}_n\right)\geq (\mu+\e)n\right)  &\leq \exp{\left(-cn^r\sharp \rmB_{K,M}\right)}\\
  &\leq \exp{(-4d\alpha \xi^r n^r)}.
}
Thus,
\al{
  &\quad \P(\rmT_n>(\mu+\xi)n)\\
  &\leq \P\left(\rmT_n>(\mu+\xi)n,~\exists v\in \rmB_{K,M}\text{ s.t. }\rmT_{\S_v}\left(v^{[1]}_n,v^{[2]}_n\right)<(\mu+\e)n\right)\\
  &\hspace{8mm}+\P\left(\forall v\in \rmB_{K,M}\text,~\rmT_{\S_v}\left(v^{[1]}_n,v^{[2]}_n\right)\geq (\mu+\e)n\right)\\
  &\leq \P\left(\rmT_n>(\mu+\xi)n,~\exists v\in \rmB_{K,M}\text{ s.t. }\rmT_{\S_v}\left(v^{[1]}_n,v^{[2]}_n\right)<(\mu+\e)n\right)+\exp{(-4d\alpha \xi^r n^r)},
  }
  where the second term is negligible.
\begin{prop}
  Suppose that $\rmT_n>(\mu+\xi)n$ and there exists $v\in \rmB_{K,M}$\text{ such that }$$\rmT_{\S_v}\left(v^{[1]}_n,v^{[2]}_n\right)<(\mu+\e)n.$$ Then there exist $x\in \rmD_M(0)$ and $y\in \rmD_M(n\mathbf{e}_1)$ such that
  $$\rmT(0,x)+\rmT(y,n\mathbf{e}_1)\geq (\xi-\e)n.$$
\end{prop}
  \begin{proof}
    Let $v\in \rmB_{K,M}$ be such that $\rmT_{\S_v}\left(v^{[1]}_n,v^{[2]}_n\right)<(\mu+\e)n$. 
     By the triangular inequality,
    \al{
      (\mu+\xi)n< \rmT_n
      &\leq \rmT(0,v^{[1]}_n)+\rmT(v^{[2]}_n,n\mathbf{e}_1)+\rmT_{\S_v}\left(v^{[1]}_n,v^{[2]}_n\right)\\
      &< \rmT(0,v^{[1]}_n)+\rmT(v^{[2]}_n,n\mathbf{e}_1)+(\mu+\e)n.
      }
   Thus,
    $$\rmT(0,v^{[1]}_n)+\rmT(v^{[2]}_n,n\mathbf{e}_1)\geq (\xi-\e)n,$$ and $x=v^{[1]}_n$ and $y=v^{[2]}_n$ are the desired objects.
  \end{proof}
  Using the proposition above,
  \aln{
    &\quad \P\left(\rmT_n>(\mu+\xi)n,~\exists v\in \rmB_{K,M}\text{ s.t. }\rmT_{\S_v}\left(v^{[1]}_n,v^{[2]}_n\right)<(\mu+\e)n\right)\nonumber\\
    &\leq \P(\exists{}x\in \rmD_M(0),~\exists y\in \rmD_M(n\mathbf{e}_1)\text{ s.t. }\rmT(0,x)+\rmT(y,n\mathbf{e}_1)\geq (\xi-\e)n)\nonumber\\
    &\leq \sum_{x\in \rmD_M(0)}\sum_{y\in \rmD_M(n\mathbf{e}_1)}\P(\rmT(0,x)+\rmT(y,n\mathbf{e}_1)\geq (\xi-\e)n).\label{comp3}
}
  To estimate the inside of the summation, for $n>8M$, we consider $4d$ disjoint paths $\{r^x_i\}^{2d}_{i=1}\subset \rmD_{2M}(0)$ from $0$ to $x$ and $\{r^y_i\}^{2d}_{i=1}\subset \rmD_{2M}(n\mathbf{e}_1)$ from $y$ to $n\mathbf{e}_1$ so that
  $$\max\{\sharp r^{z}_i:~i\in\{1,\cdots,2d\},~z\in\{x,y\}\}\leq 4dM,$$
  where $\sharp r$ is the number of edges in a path $r$ as in \cite[p 135]{Kes86}. Then,
\al{
  \P(\rmT(0,x)+\rmT(y,n\mathbf{e}_1)\geq (\xi-\e)n) &\leq \P(\forall i\in\{1,\cdots,2d\},~\rmT(r_i^x)+\rmT(r^y_i)\geq (\xi-\e)n)\nonumber\\
  &= \prod_{i=1}^{2d}\P\left(\sum_{e\in r_i^x\cup r_i^y}\tau_e \geq (\xi-\e)n\right)\nonumber\\
  &\leq \exp{(-2d(1-\e)\alpha ((\xi-\e) n)^r )},
}
where we have used Lemma~\ref{exp:est} below in the last line.
\begin{lem}\label{exp:est}
  Let $(X_i)_{i=1}^k$ be identical and independent distrbutions satisfying \eqref{cond:distr}. Then for any $c>0$ there exists $n_0=n_0(k,c)$ such that for any $n\geq n_0$,
  $$\P\left(\sum_{i=1}^k X_i>n\right)\leq  \exp{(-(1-c)\alpha n^r)}.$$
\end{lem}
\begin{proof}
  Since $\E \exp{(\beta X_1^r)}<\infty$ for $\beta<\alpha$ and
  \ben{\label{concave}
    \left(\displaystyle\sum_{i=1}^k x_i\right)^r\leq \displaystyle\sum_{i=1}^k x_i^r\text{ for $x_i\geq 0$},
    }by the exponential Markov inequatliy, for sufficiently large $n$,
  \al{
    \P\left(\sum_{i=1}^k X_i>n\right)&\leq \exp{(-(1-c/2)\alpha n^r)}(\E \exp{((1-c/2) \alpha X_1^r)})^k\\
    &\leq \exp{(-(1-c)\alpha n^r)}.
    }
\end{proof}
Therefore, \eqref{comp3} can be bounded from above by
\al{
  &\quad \sum_{x\in \rmD_M(0)}\sum_{y\in \rmD_M(n\mathbf{e}_1)}\exp{(-2d(1-\e)\alpha((\xi-\e) n)^r)}\\
  &\leq (4M)^{2d}\exp{(-2d(1-\e)\alpha((\xi-\e) n)^r)}\\
  &\leq \exp{(-(2d\alpha\xi^r +o_\e(1)) n^r)},
}
where $o_\e(1)$ is a positive constant depending on $\e$, which converges to $0$ as $\e\to 0$ after $n\to\infty$. Since $\e$ is arbitrary, letting $\e\to 0$ after $n\to\infty$, we get
$$\limsup_{n\to\infty}\frac{1}{n^r}\log{\P(\rmT_n>(\mu+\xi)n)}\leq -2d\alpha \xi^r.$$
\subsection{Proof of Theorem~\ref{thm:main3}}
 By \eqref{anomaly:distr}, it is straightforward to check that there exist $c_4,c_5>0$ such that for sufficiently large $t$,
$$c_4\exp{(-\alpha_2t)}\leq \P(\tau_e>t)\leq c_5\exp{(-\alpha_1t)}.$$
Hence, the same argument as in Theorem~\ref{thm:main1} shows
\bean{\label{anomaly:cons}
  -2d\alpha_2\xi&\leq&\liminf_{n\to\infty}\frac{1}{n}\log{\P(\rmT(0,n\mathbf{e}_1)>n(\mu(\mathbf{e}_1)+\xi))}\\
  &\leq& \limsup_{n\to\infty}\frac{1}{n}\log{\P(\rmT(0,n\mathbf{e}_1)>n(\mu(\mathbf{e}_1)+\xi))}\leq -2d\alpha_1\xi.\nonumber
  }
We start with proving \eqref{limsup}. We consider $\P(\rmT_{a_{2n-1}^2}>(\mu+\xi)a_{2n-1}^2)$. Then, for sufficiently large $n$ such that $\xi\ll a_{2n-1}$, the first term of \eqref{lower-est} turns out to be
\al{
  \P(\forall e\in \E_1,~\tau_e>(\xi+\e)a_{2n-1}^2)&\geq \left(\int_{(\xi+\e)a_{2n-1}^2}^{a_{2n}}c_3\exp{(-\alpha_1 x)}dx\right)^{2d}\\
  &\geq \left(\int_{(\xi+\e)a_{2n-1}^2}^{(\xi+\e)a_{2n-1}^2+1}c_3\exp{(-\alpha_1 x)}dx\right)^{2d}\\
  &\geq \left(c_3\exp{(-\alpha_1 ((\xi+\e)a_{2n-1}^2+1))}\right)^{2d}\\
  &= c_3^{2d}\exp{(-2d\alpha_1 ((\xi+\e)a_{2n-1}^2+1))}.
}
Since the second term of \eqref{lower-est} converges to $1$ as before, letting $\e\to 0$ finally, we have
\al{
  \limsup_{n\to\infty}\frac{1}{n}\log{\P(\rmT_{n}>(\mu+\xi)n)}&\geq \limsup_{n\to\infty} \frac{1}{a_{2n-1}^2}\log{\P(\rmT_{a_{2n-1}^2}>(\mu+\xi)a_{2n-1}^2)}\\
  &\geq -2d\alpha_1 \xi. 
}
To prove \eqref{liminf}, we consider $\P(\rmT_{a_{2n}^2}>(\mu+\xi)a_{2n}^2)$. Since $\E[\exp{(-\rho \tau_e)}]<\infty$ for $\rho<\alpha_1$, \allowbreak \cite[Lemma 3.1]{CZ03} proves Lemma~\ref{lem:Zhang} with $r=1$. Lemma~\ref{exp:est} is replaced by the following.
\begin{lem}\label{anomaly:lem}
  Let $(X_i)_{i=1}^k$ be identical and independent distrbutions satisfying \eqref{anomaly:distr}. Then for any $c>0$ there exists $n_0=n_0(k,c)$ such that for any $n\geq n_0$,
  $$\P\left(\sum_{i=1}^k X_i> (\xi-\e)a_{2n}^2 \right)\leq  \exp{(-(1-c)\alpha_2(\xi-\e) a_{2n}^2)}.$$
\end{lem}
\begin{proof}
  By the union bound,
  \al{
    &\quad \P\left(\sum_{i=1}^k X_i> (\xi-\e)a_{2n}^2 \right)\\
    &\leq \P\left(\sum_{i=1}^k X_i\mathbf{1}_{\{X_i<a_{2n+1}\}}> (\xi-\e)a_{2n}^2 \right)+\P\left(\exists i\in\{1,\cdots,k\}\text{ s.t. }X_i\geq a_{2n+1}\right).
  }
  Since $a_{2n+1}=2^{a_{2n}}$, the second term can be bounded from above by $\exp{(-c2^{a_{2n}})}$ with some $c>0$, which is negligible. For the first term,
  \al{
    &\quad \P\left(\sum_{i=1}^k X_i\mathbf{1}_{\{X_i<a_{2n+1}\}}> (\xi-\e)a_{2n}^2 \right)\\
    &= \P\left(\sum_{i=1}^k X_i\mathbf{1}_{\{a_{2n}\leq X_i<a_{2n+1}\}}+\sum_{i=1}^k X_i\mathbf{1}_{\{X_i<a_{2n}\}}> (\xi-\e)a_{2n}^2 \right)\\
    &\leq  \P\left(\sum_{i=1}^k X_i\mathbf{1}_{\{a_{2n}\leq X_i<a_{2n+1}\}}+k a_{2n}> (\xi-\e)a_{2n}^2 \right)\\
    &=  \P\left(\sum_{i=1}^k X_i\mathbf{1}_{\{a_{2n}\leq X_i<a_{2n+1}\}}> (\xi-\e)a_{2n}^2-k a_{2n} \right).
  }
 Since for $\beta<\alpha_2$, $\E[\exp{(-\beta X_i\mathbf{1}_{\{a_{2n}\leq X_i<a_{2n+1}\}})}]$ is uniformly bounded for $n$, by the exponential Markov inequality and \eqref{anomaly:distr}, for sufficiently large $n$, this is further bounded from above by
  \al{
    &\quad \exp{\left(-\left(1-\frac{c}{2}\right)\alpha_2((\xi-\e)a_{2n}^2-k a_{2n})\right)} \E\left[\exp{\left(- \left(1-\frac{c}{2}\right)\alpha_2 X_i\mathbf{1}_{\{a_{2n}\leq X_i<a_{2n+1}\}}\right)}\right]^k\\
    &\leq \exp{\left(-\left(1-c\right)\alpha_2(\xi-\e)a_{2n}^2\right)}.
    }
\end{proof}
The same computation with Lemma~\ref{anomaly:lem} as in Theorem~\ref{thm:main1} shows
$$\P(\rmT_{a_{2n}^2}>(\mu+\xi)a_{2n}^2)\leq \exp{(-(2d\alpha_2 \xi +o_\e(1)) n)}+\exp{(-4d\alpha_2 \xi a_{2n}^2)},$$
where $o_{\e}(1)$ is a constant depending on $\e$, which converges to $0$ as $\e\to 0$ after $n\to\infty$. Thus, letting $\e\to 0$ after $n\to\infty$, we obtain
\al{
  \liminf_{n\to\infty}\frac{1}{n}\log{\P(\rmT_{n}>(\mu+\xi)n)}&\leq \liminf_{n\to\infty} \frac{1}{a_{2n}^2}\log{\P(\rmT_{a_{2n}}>(\mu+\xi)a_{2n}^2)}\\
  &\leq -2d\alpha_2\xi. 
}
Together with \eqref{anomaly:cons}, the proof is completed.
\section{Appendix}
\subsection{Proof of Lemma~\ref{lem:Zhang}}
In this section, we show Lemma~\ref{lem:Zhang}: for any $\e>0$, there exist $M=M(\e)\in\N$ and a positive constant $c=c(\e,M)$ such that $n\geq M$,
  $$\P\left(\rmT_{[0,n]\times[-M,M]^{d-1}}\left(0,n\mathbf{e}_1\right)\geq (\mu+\e)n\right)\leq e^{-cn^r}.$$
\begin{proof}
The proof is based on that in \cite[Lemma 3.1]{CZ03}. 
 By \eqref{kingman}, for any $s,\e>0$, there exist $K<M\in\N$ such that
\ben{\label{LLN:box}
\P\left(\rmT_{[-M,M]^d}\left(0,K\mathbf{e}_1\right)\geq \mu\left(1+\frac{\e}{2}\right)K\right)< s.
}
Given $\e>0$, let
$$s=s(\e)=\frac{\e^2}{16\E\tau_e^2},$$
and $K=K(\e,s)$, $M=M(\e,s)$ in \eqref{LLN:box}. For $n\in\N$, we write
$$\S^M_n=[0,n]\times[-M,M]^{d-1}.$$
Let $n>KM$. If we take $\ell=\lf n/(KM)\rf$, then by the triangular inequality,
\al{
  &\quad \P\left(\rmT_{\S_n^M}\left(0,n\mathbf{e}_1\right)\geq (\mu+\e)n\right)\\
    &\leq \P\left(\rmT_{\S^M_{KM\ell}}\left(0,(KM) \ell \mathbf{e}_1\right)+\rmT_{\S_n^M}\left((KM)\ell \mathbf{e}_1,n\mathbf{e}_1\right)\geq \left(\mu+\frac{\e}{2}\right)KM \ell +\frac{\e}{2}n\right)\\
  &\leq \P\left(\rmT_{\S^M_{KM\ell}}\left(0,(KM) \ell \mathbf{e}_1\right)\geq \left(\mu+\frac{\e}{2}\right)KM \ell\right)+\P\left(\rmT_{\S_n^M}\left((KM)\ell \mathbf{e}_1,n\mathbf{e}_1\right)\geq \frac{\e}{2}n\right),
  }
and, by Lemma~\ref{exp:est}, the second term can be bounded from above by $e^{-c n^r}$ with some $c>0$.
Hence, without loss of generalty, We can assume $n$ is divisible by $KM$, say $n=KM\ell$ with $\ell\in\N$. Given $i\in\{0,\cdots,M\ell-1\}$, we define
$$\rmT^{[i]}=\rmT_{(iK\mathbf{e}_1+[-M,M]^d)}(iK\mathbf{e}_1,(i+1)K\mathbf{e}_1).$$
It follows from the definition and the triangular inequality that if $|i-j|> M$, then $\rmT^{[i]}$ and $\rmT^{[j]}$ are independent and
$$\rmT_{\S_n^M}\left(0,n\mathbf{e}_1\right)\leq \sum_{i=0}^{M\ell-1} \rmT^{[i]}+\rmT_{\S_n^M}(0,KM\mathbf{e}_1)+\rmT_{\S_n^M}((n-KM)\mathbf{e}_1,n\mathbf{e}_1).$$
Note that by Lemma~\ref{exp:est},
$$\P(\rmT_{\S_n^M}(0,KM\mathbf{e}_1)\geq \e n)+\P(\rmT_{\S_n^M}((n-KM)\mathbf{e}_1,n\mathbf{e}_1)\geq \e n)\leq e^{-cn^r},$$
with some $c>0$. Thus, by the union bound and $n=KM\ell$,
\al{
  &\quad \P\left(\rmT_{\S_n^M}\left(0,n\mathbf{e}_1\right)\geq (\mu+3\e)n\right)\\
  &\leq \P\left( \sum_{i=0}^{M\ell-1} \rmT^{[i]}\geq (\mu+\e)n\right)+\P(\rmT_{\S_n^M}(0,KM\mathbf{e}_1)\geq \e n)+\P(\rmT_{\S_n^M}((n-KM)\mathbf{e}_1,n\mathbf{e}_1)\geq \e n)\\
  &\leq  \P\left( \sum_{m=0}^{M-1}\sum_{i=0}^{\ell-1} \rmT^{[Mi+m]}\geq (\mu+\e)KM\ell \right)+e^{-cn^r}\\
  &\leq  \sum_{m=0}^{M-1}\P\left( \sum_{i=0}^{\ell-1} \rmT^{[Mi+m]}\geq (\mu+\e)K\ell \right)+e^{-cn^r}.
}
For the first term, we only consider the case $m=0$, since the other cases can be treated in the same way. Then,
\aln{
  &\quad \P\left( \sum_{i=0}^{\ell-1} \rmT^{[Mi]}\geq (\mu+\e)K\ell\right)\nonumber\\
  & = \P\left( \sum_{i=0}^{\ell-1} \rmT^{[Mi]}\mathbf{1}_{\{\rmT^{[Mi]}< (\mu+(\e/2))K\}}+\sum_{i=0}^{\ell-1} \rmT^{[Mi]}\mathbf{1}_{\{\rmT^{[Mi]}\geq (\mu+(\e/2))K\}}\geq (\mu+\e)K\ell\right)\nonumber\\
  &\leq \P\left( \sum_{i=0}^{\ell-1}  (\mu+(\e/2))K+\sum_{i=0}^{\ell-1} \rmT^{[Mi]}\mathbf{1}_{\{\rmT^{[Mi]}\geq (\mu+(\e/2))K\}}\geq (\mu+\e)K\ell\right)\nonumber\\
  &= \P\left( \sum_{i=0}^{\ell-1} \rmT^{[Mi]}\mathbf{1}_{\{\rmT^{[Mi]}\geq (\mu+(\e/2))K\}}\geq \e K\ell/2\right).\label{comp}
}
For $m\in\N$, let us denote
$${\rm e}{[m]}=\la m\mathbf{e}_1,(m+1)\mathbf{e}_1\ra.$$
Then, since
$$\rmT^{[Mi]}\leq \displaystyle\sum_{m=KMi}^{KMi+K-1}\tau_{{\rm e}{[m]}},$$ \eqref{comp} can be bounded from above by
\al{
  &\quad \P\left(\sum_{i=0}^{\ell-1}\sum_{m=KMi}^{KMi+K-1}\tau_{{\rm e}{[m]}}\mathbf{1}_{\{\rmT^{[Mi]}\geq (\mu+(\e/2))K\}} \geq \e K\ell/2\right)\\
  &=\P\left(\sum_{m=0}^{K-1}\sum_{i=0}^{\ell-1}\tau_{{\rm e}{[KMi+m]}}\mathbf{1}_{\{\rmT^{[Mi]}\geq (\mu+(\e/2))K\}} \geq \e K\ell/2\right)\\
  &\leq \sum_{m=0}^{K-1} \P\left(\sum_{i=0}^{\ell-1}\tau_{{\rm e}{[KMi+m]}}\mathbf{1}_{\{\rmT^{[Mi]}\geq (\mu+(\e/2))K\}} \geq \e\ell /2\right).
  }
We write $X^m_i=\tau_{{\rm e}{[KMi+m]}}\mathbf{1}_{\{\rmT^{[Mi]}\geq (\mu+(\e/2))K\}}$. Then $(X^m_i)_i$ are identically and independent non-negative random variables. By \eqref{LLN:box} and the Cauchy--Schwarz inequality,
\al{\E X^m_i&\leq (\E\tau_{e}^2)^{1/2}\P\left(\rmT^{[Mi]}\geq (\mu+(\e/2))K\right)^{1/2}\\
  &< (\E\tau_{e}^2)^{1/2} s^{1/2}= \e/4,
}
and for sufficiently large $t>0$, $\P(X^m_i>t)\leq \P(\tau_e>t)\leq c_2e^{-\alpha t^r},$ where $c_2$ is in \eqref{cond:distr}. Thus by \cite[(4.2)]{GRR}, there exists $c=c(\e,K,M)>0$ such that
$$\P\left(\sum_{i=0}^{\ell-1}X^m_i \geq \e\ell /2\right) \leq e^{-cn^r}.$$
Putting things together, the proof of Lemma~\ref{lem:Zhang} is completed.
\end{proof}
\subsection{Proof of \eqref{thm:main2}}\label{slow-vary}
The proof is essentially the same as in Theorem~\ref{thm:main1}. We only touch with the difference. Since $c_1(t),{\rm b}_1(t)$ are slowly varying, for sufficiently large $n$, the first term of \eqref{lower-est} is bounded from below by
$$c_1((\xi+\e)n)^{2d}\exp{(-2d {\rm b}_1((\xi+\e)n)((\xi+\e)n)^r)}\geq \exp{(-2d (1+\e) {\rm b}_1(n)((\xi+\e)n)^r)}.$$
Thus, \eqref{lower-main} is replaced by
$$  \P(\rmT_n>(\mu+\xi)n)\geq \frac{1}{2}\exp{(-2d (1+\e) {\rm b}_1(n)((\xi+\e)n)^r))}.$$
For the lower bound, the rest is the same as before. For the upper bound, the proof is exactly the same as in Theorem~\ref{thm:main1} except for Lemma~\ref{lem:Zhang} and Lemma~\ref{exp:est}. Lemma~\ref{exp:est} is replaced by the following lemma.

\begin{lem}\label{exp:est2}
  Let $(X_i)_{i=1}^k$ be identical and independent distrbutions satisfying \eqref{cond-distr2}. Then for any $c>0$ there exists $n_0=n_0(k,c)$ such that for any $n\geq n_0$,
  $$\P\left(\sum_{i=1}^k X_i>n\right)\leq  \exp{(-(1-c){\rm b}_2(n) n^r)}.$$
\end{lem}
\begin{proof}
  Since $c_2(t),{\rm b}_2(t)$ are  slowly varying, by the local uniformity of slowly varying functions, there exists $t_1=t_1(c,k)>0$ such that for any $t>t_1$ and $t'\in \left[\frac{c}{4k}t,t\right]$,
  \ben{\label{slow:vary}
    c_2\left(t'\right)\leq 2c_2(t),~{\rm b}_2\left(t'\right)\geq \left(1-\frac{c}{4}\right){\rm b}_2(t).
    }
  Moreover, if $\displaystyle\sum_{1\leq i\leq k}  x'_i>n$ for $x'_i\in\R_{\geq 0}$, then there exist $\ell>n-k$ and $x_i\in\Z_{\geq 0}$ such that $x'_i\geq x_i$ and $\sum_i x_i=\ell$. In fact, we can take $x_i=\lf x'_i\rf$ and $\ell=\sum_i x_i$. Thus,
  \aln{
    \P\left(\sum_{i=1}^kX_i>n\right)&\leq \sum_{\ell\geq n-k}\sum_{\sum_i x_i=\ell}\P(X_i\geq x_i~\forall i)\nonumber\\
    &=\sum_{\ell\geq n-k}\sum_{\sum_i x_i=\ell}\prod_{i=1}^k\P(X_i\geq x_i)\nonumber\\
    &\leq  \sum_{\ell\geq n-k}\sum_{\sum_i x_i=\ell}\prod_{i{\rm :}x_i\geq \frac{c\ell}{4k}}\P(X_i\geq x_i),\nonumber
  }
  where the second sum runs over all $(x_i)_{i=1}^k\subset \Z_{\geq 0}$ such that $\sum^k_{i=1} x_i=\ell$ and the product in the last line runs over all $i\in\{1,\cdots,k\}$ such that $x_i\geq \frac{c\ell}{4k}$. If $n-k>t_1$, by \eqref{slow:vary}, this is further bounded from above by
  \aln{
    &\quad  \sum_{\ell\geq n-k}\sum_{\sum_i x_i=\ell} \prod_{i{\rm :}x_i\geq \frac{c\ell}{4k}}c_2(x_i)\exp{(-{\rm b}_2(x_i)x_i^r)}\nonumber\\
    &\leq  \sum_{\ell\geq n-k}\sum_{\sum_i x_i=\ell}2^{k}(c_2(\ell)+ 1)^k\exp{\left(-\left(1-\frac{c}{4}\right){\rm b}_2(\ell)\sum_{i{\rm :}x_i\geq \frac{c\ell}{4k}}x_i^r\right)}.\label{comp1}
  }
 For any $\ell\in\N$ and $\displaystyle(x_i)_{i=1}^k\subset \Z_{\geq 0}$ with $\displaystyle\sum^k_{i=1}x_i=\ell$,
  \ben{\label{comp2}
    \sum_{i{\rm :}x_i\geq \frac{c\ell}{4k}}x_i=\ell-\sum_{i{\rm :}x_i<\frac{c\ell}{4k}}x_i\geq (1-(c/4))\ell.
    }  By \eqref{concave} and \eqref{comp2}, for sufficiently large $n$, \eqref{comp1} can be further bounded from above by
  \al{
    &\quad  \sum_{\ell\geq n-k}\sum_{\sum_i x_i=\ell}2^{k}(c_2(\ell)+ 1)^k\exp{\left(-\left(1-\frac{c}{4}\right){\rm b}_2(\ell)\left(\sum_{i{\rm :}x_i\geq \frac{c\ell}{4k}}x_i\right)^r\right)} \\
    &\leq \sum_{\ell\geq n-k}\sum_{\sum_i x_i=\ell}2^{-k}(c_2(\ell)+ 1)^{k}\exp{\left(-\left(1-\frac{c}{4}\right)\left(1-\frac{c}{4}\right){\rm b}_2(\ell)\ell^r\right)}\\
    &\leq \sum_{\ell\geq n-k}\ell^k 2^{-k}(c_2(\ell)+ 1)^{k}\exp{\left(-\left(1-\frac{c}{2}+\frac{c^2}{16}\right){\rm b}_2(\ell)\ell^r\right)}\\
    &\leq \sum_{\ell\geq n-k} \exp{\left(-\left(1-\frac{c}{2}\right){\rm b}_2(\ell)\ell^r\right)}\leq \exp{\left(-(1-c){\rm b}_2(n)n^r\right)},
  }
  where we have used ${\rm b}_2(\ell)\ell^r\geq (\ell/n)^{r/2} {\rm b}_2(n)n^r$ for $\ell>2n$ and $\left|\frac{{\rm b}_2(t)}{{\rm b}_2(n)}-1\right|\ll 1$ for \linebreak$t\in[n-k,2n]$ in the last line.  
\end{proof}
Lemma~\ref{lem:Zhang} is replaced by the following. 
\begin{lem}
For any $\e>0$, there exist $M=M(\e)\in\N$ and a positive constant $c=c(\e,M)$ such that $n\geq M$,
$$\P\left(\rmT_{[0,n]\times[-M,M]^{d-1}}\left(0,n\mathbf{e}_1\right)\geq (\mu+\e)n\right)\leq e^{-c {\rm b}_2(n) n^r}.$$
\end{lem}
Using Lemma~\ref{exp:est2} instead of Lemma~\ref{exp:est}, the proof is the same as in Lemma~\ref{lem:Zhang}, so we omit this. The rest is the same as before.
\section*{Acknowledgements}
The author would like to thank Ryoki Fukushima introducing \cite[Lemma~3.1]{CZ03}. This research is partially supported by JSPS KAKENHI	19J00660.

\begin{thebibliography}{99}
    \bibitem{ADH}
  A. Auffinger, M. Damron, J. Hanson,
\newblock 50 years of first-passage percolation. 
\newblock {\em University Lecture Series, 68. American Mathematical Society, Providence, RI}, 2017. \MR{3729447}

 \bibitem{BGS}
  R. Basu, S. Ganguly, A. Sly,
\newblock Upper Tail Large Deviations in First Passage Percolation.
\newblock ArXiv e-print 1712.01255.

 \bibitem{CGM09}
M.Cranston, D. Gauthier, and T.S.Mountford. \newblock On large deviation regimes for random media models.
{\em Ann. Appl. Probab.} 826-862, 2009. \MR{2521889}

\bibitem{CKN} V. H. Can, N. Kubota, S. Nakajima. {\it Large deviations of the first passage time in the frog model}, in preparation.
  
\bibitem{CZ03}
  Y. Chow, Y. Zhang. \newblock Large deviations in first-passage percolation. {\em Ann. Appl. Probab.} Volume 13, Number 4, pp 1601-1614, 2003. \MR{2023891}

  \bibitem{Eden}
    Murray Eden. \newblock A two-dimensional growth process. {\em Proceedings of Fourth Berkeley Symposium on Mathematics, Statistics, and Probability.} 4. Berkeley: University of California Press. pp. 223-239, 1961. \MR{0136460}
    
  \bibitem{GRR}
    N. Gantert, K. Ramanan, F. Rembart. \newblock Large deviations for weighted sums of stretched exponential random variables Electron. {\em Commun. Probab}. 114, 1-14, 2014. \MR{3233203}
      \bibitem{HW65}
    J. M. Hammersley and D. J. A. Welsh, \newblock First-passage percolation, subadditive processes, stochastic networks and generalized renewal theory, in Bernoulli, Bayes, Laplace Anniversary Volume (J. Neyman and L. Lecam, eds.), \newblock {\em Springer-Verlag, Berlin and New York.} 61-110, 1965. \MR{198576}
\bibitem{Kes86}
  Harry Kesten. \newblock Aspects of first passage percolation. In Lecture Notes in Mathematics.  vol. 1180, pp 125-264, 1986. \MR{0876084}
  
 \bibitem{King73} J. F. C. Kingman. \newblock Subadditive Ergodic Theory. {\rm Annals of Probab.} 1, 883-899, 1973. \MR{356192}

\bibitem{KPZ86} M. Kardar, G. Parisi, Y.-C. Zhang, Dynamic scaling of growing interfaces. {\em Phys. Rev. Lett.} 56 889-892, 1986.
      \bibitem{KS91}
H. Krug and H. Spohn.
\newblock Kinetic roughening of growing surfaces. In:
\newblock {\em Solids Far From Equilibrium.}
\newblock C.Godre\`che ed., Cambridge University Press, 1991.
\end{thebibliography}

\end{document}